\documentclass[12pt]{amsart}
\usepackage{amscd,amsmath,amsthm,amssymb,graphics}
\usepackage{pstcol,pst-plot,pst-3d}
\usepackage{lmodern,pst-node}
\usepackage{multicol}
\usepackage{epic,eepic}
\usepackage{amsfonts,amssymb,amscd,amsmath,enumerate,verbatim}
\psset{unit=0.7cm,linewidth=0.8pt,arrowsize=2.5pt 4}

\newpsstyle{fatline}{linewidth=1.5pt}
\newpsstyle{fyp}{fillstyle=solid,fillcolor=verylight}
\definecolor{verylight}{gray}{0.97}
\definecolor{light}{gray}{0.9}
\definecolor{medium}{gray}{0.85}
\definecolor{dark}{gray}{0.6}

\def\NZQ{\Bbb}               

\def\ZZ{{\NZQ Z}}

\def\FF{{\NZQ F}}

\def\frk{\frak}               

\def\pp{{\frk p}}

\def\mm{{\frk m}}

\def\Phi{{\frk n}}
\def\Phi{{\frk N}}
\def\MI{{\mathcal I}}

\def\MR{{\mathcal R}}

\def\MT{{\mathcal T}}

\def\MS{{\mathcal S}}

\def\opn#1#2{\def#1{\operatorname{#2}}} 
\opn\chara{char} \opn\length{\ell} \opn\pd{pd} \opn\rk{rk}
\opn\projdim{proj\,dim} \opn\injdim{inj\,dim} \opn\rank{rank}
\opn\depth{depth} \opn\grade{grade} \opn\height{height}
\opn\embdim{emb\,dim} \opn\codim{codim}

\opn\Tr{Tr} \opn\bigrank{big\,rank}
\opn\superheight{superheight}\opn\lcm{lcm}
\opn\trdeg{tr\,deg}
\opn\reg{reg} \opn\lreg{lreg} \opn\ini{in} \opn\lpd{lpd}
\opn\size{size}\opn\bigsize{bigsize}
\opn\cosize{cosize}\opn\bigcosize{bigcosize}
\opn\sdepth{sdepth}\opn\sreg{sreg}
\opn\link{link}\opn\fdepth{fdepth}
\opn\div{div} \opn\Div{Div} \opn\cl{cl} \opn\Cl{Cl}
\opn\Spec{Spec} \opn\Supp{Supp} \opn\supp{supp} \opn\Sing{Sing}
\opn\Ass{Ass} \opn\Min{Min}\opn\Mon{Mon} \opn\dstab{dstab} \opn\astab{astab}
\opn\Syz{Syz}
\opn\Ann{Ann} \opn\Rad{Rad} \opn\Soc{Soc}
\opn\Im{Im} \opn\Ker{Ker} \opn\Coker{Coker} \opn\Am{Am}
\opn\Hom{Hom} \opn\Tor{Tor} \opn\Ext{Ext} \opn\End{End}
\opn\Aut{Aut} \opn\id{id}

\opn\nat{nat}
\opn\pff{pf}
\opn\Pf{Pf} \opn\GL{GL} \opn\SL{SL} \opn\mod{mod} \opn\ord{ord}
\opn\Gin{Gin} \opn\Hilb{Hilb}\opn\sort{sort}
\opn\initial{init}
\opn\ende{end}
\opn\height{height}
\opn\type{type}
\opn\aff{aff} \opn\con{conv} \opn\relint{relint} \opn\st{st}
\opn\lk{lk} \opn\cn{cn} \opn\core{core} \opn\vol{vol}
\opn\link{link} \opn\star{star}\opn\lex{lex}
\opn\gr{gr}

\def\pot#1#2{#1[\kern-0.28ex[#2]\kern-0.28ex]}

\opn\dirlim{\underrightarrow{\lim}}
\opn\inivlim{\underleftarrow{\lim}}
\let\union=\cup
\let\sect=\cap

\let\tensor=\otimes
\let\iso=\cong

\let\Dirsum=\bigoplus

\let\to=\rightarrow

\def\Implies{\ifmmode\Longrightarrow \else
        \unskip${}\Longrightarrow{}$\ignorespaces\fi}
\def\implies{\ifmmode\Rightarrow \else
        \unskip${}\Rightarrow{}$\ignorespaces\fi}
\def\iff{\ifmmode\Longleftrightarrow \else
        \unskip${}\Longleftrightarrow{}$\ignorespaces\fi}

\let\:=\colon
\newtheorem{Theorem}{Theorem}[section]
 \newtheorem{Lemma}[Theorem]{Lemma}
 \newtheorem{Corollary}[Theorem]{Corollary}
 
 \newtheorem{Remark}[Theorem]{Remark}

\let\epsilon\varepsilon
\let\kappa=\varkappa
\def\qed{\ifhmode\textqed\fi
      \ifmmode\ifinner\quad\qedsymbol\else\dispqed\fi\fi}
\def\textqed{\unskip\nobreak\penalty50
       \hskip2em\hbox{}\nobreak\hfil\qedsymbol
       \parfillskip=0pt \finalhyphendemerits=0}
\def\dispqed{\rlap{\qquad\qedsymbol}}

\opn\dis{dis}
\def\pnt{{\raise0.5mm\hbox{\large\bf.}}}

\opn\Lex{Lex}

\begin{document}
 \title{Pseudo-Gorenstein and level Hibi rings}

 \author {Viviana Ene, J\"urgen Herzog, Takayuki Hibi  and Sara Saeedi Madani}

\address{Viviana Ene, Faculty of Mathematics and Computer Science, Ovidius University, Bd.\ Mamaia 124,
 900527 Constanta, Romania, and
 \newline
 \indent Simion Stoilow Institute of Mathematics of the Romanian Academy, Research group of the project  ID-PCE-2011-1023,
 P.O.Box 1-764, Bucharest 014700, Romania} \email{vivian@univ-ovidius.ro}

\address{J\"urgen Herzog, Fachbereich Mathematik, Universit\"at Duisburg-Essen, Campus Essen, 45117
Essen, Germany} \email{juergen.herzog@uni-essen.de}

\address{Takayuki Hibi, Department of Pure and Applied Mathematics, Graduate School of Information Science and Technology,
Osaka University, Toyonaka, Osaka 560-0043, Japan}
\email{hibi@math.sci.osaka-u.ac.jp}

\address{Sara Saeedi Madani, Department of Pure Mathematics, Faculty of Mathematics and Computer Science, Amirkabir University
of Technology (Tehran Polytechnic), 424, Hafez Ave., Tehran 15914, Iran, and
\newline
 \indent School of Mathematics,
Institute for Research in Fundamental Sciences (IPM), P.O. Box 19395-5746, Tehran, Iran} \email{sarasaeedi@aut.ac.ir}

\thanks{The first author was supported by the grant UEFISCDI,  PN-II-ID-PCE- 2011-3-1023.}
\thanks{The paper was written while the forth author was visiting the Department of Mathematics of University Duisburg-Essen. She wants to express her thanks for its hospitality.}

 \begin{abstract}
 We introduce pseudo-Gorenstein rings and characterize those Hibi rings attached to a finite distributive lattice $L$  which are pseudo-Gorenstein. The characterization is given  in terms of the poset of join-irreducible elements of $L$. We also present a necessary condition for Hibi rings to be level. Special attention is given to planar and hyper-planar lattices.   Finally the pseudo-Goresntein and level property of Hibi rings and generalized Hibi rings
 is compared with each other.
 \end{abstract}

\thanks{}

 \maketitle

\section*{Introduction}
\label{introduction}
Let $K$ be a field. Naturally attached to a finite distributive lattice $L$ is a $K$-algebra $K[L]$ which nowadays is called the Hibi ring of $L$. This $K$-algebra was introduced by the third author in 1987, see \cite{H}. In that paper it is shown that $K[L]$ is a normal Cohen--Macaulay domain and that $K[L]$ is Gorenstein if and only if the  poset $P$ of join-irreducible elements of $L$ is pure.

Let $R$ be an arbitrary standard graded Cohen--Macaulay  $K$-algebra with canonical module $\omega_R$. Then $R$ is Gorenstein if and only if $\omega_R$ is a cyclic module, and hence generated in a single degree. The condition on $\omega_R$ may be weakened in different ways. If one only requires that the generators of $\omega_R$ are of the same  degree, then $R$ is called a level ring, and if one requires that there is only one generator of least degree, then we call $R$ a pseudo-Gorenstein ring.

In this paper we intend to characterize the pseudo-Gorenstein and level Hibi rings $K[L]$ in terms of $P$.  For that purpose  we use the basic fact, observed in \cite{H}, that a $K$-basis of $K[L]$ can be described in terms of order reversing functions $\hat{P}\to \ZZ_{\geq 0}$, and that a $K$-basis of the canonical module $\omega_L$ of $K[L]$ can be described in terms of strictly order reversing functions $\hat{P}\to \ZZ_{\geq 0}$. Here $\hat{P}=P\union\{-\infty,\infty\}$ with $-\infty<x<\infty$ for all $x\in P$, see Section~\ref{pseudo} for details.  Since the property of $K[L]$ to be level, Gorenstein or pseudo-Gorenstein does not depend on the field $K$, we simply say that $L$ is level, Gorenstein or pseudo-Gorenstein if $K[L]$ has this property.

In Section~\ref{1} we briefly list conditions which are equivalent to pseudo-Gorenstein and describe the relation of this notion to that of level and Gorenstein. Since pseudo-Gorenstein rings can be identified by the property that the leading coefficient of the numerator polynomial of the Hilbert series is equal to $1$, pseudo-Gorenstein rings are much easier accessible than level rings. Theorem~\ref{classification} gives a full characterization of Hibi rings which are pseudo-Gorenstein. Indeed, it is shown that $L$ is pseudo-Gorenstein if and only if $\depth(x)+\height(x)=\rank \hat{P}$ for all $x\in P$. This is equivalent to say, that for any given $x\in P$ there exists  a  chain of maximal length in $\hat{P}$ passing through $x$. It may be of interest to notice that this property of the chains in $P$ has its analogue in the fact that in an  affine domain $\dim R=\dim R/\pp+\height \pp$ for all $\pp\in \Spec R$.

Though Theorem~\ref{classification} characterizes Hibi rings which are pseudo-Gorenstein, it may  nevertheless be  difficult to apply this characterization efficiently, even for planar  lattices. In Section~\ref{hyper} we introduce hyper-planar lattices which represent a natural extension of planar  lattices to higher dimensions. They are defined by the property that their poset $P$ of join-irreducible elements admits a canonical chain decomposition, that is, a decomposition into pairwise disjoint maximal chains.  In general such a decomposition is not unique. For hyper-planer lattices we introduce a regularity condition with the effect that the height of an element in $P$ is the same as the height of the element in the chain to which it belongs.  Apart from a few exceptions we keep this regularity hypothesis on hyper-planar lattices throughout the rest of the paper. In Theorem~\ref{equallength} it is shown that a regular hyper-planar lattice is pseudo-Gorenstein if and only if all chains in a canonical chain decomposition of $P$ have the same length. For simple planar lattices it is shown in Theorem~\ref{viviana} that the regularity condition is in fact indispensable. Indeed, it is proved  that a simple planar lattice is pseudo-Gorenstein if and only if it is regular and the (two) chains in a canonical chain decomposition of $P$ have the same length. Unfortunately this result cannot be extended to hyper-planar lattices as we show by an example.

The study of level Hibi rings is more difficult. There is a very nice sufficient condition on $P$ that guarantees that $L$ is level. In \cite[Theorem~3.3]{M} Miyazaki showed that $L$ is level if for all $x\in P$ all chains in $\hat{P}$ ascending from $x$ have the same length, and he showed by an example that this condition is not necessary. Let $P^\vee$ be the dual poset of $P$, i.e., the poset on the same set as $P$ but with all order relation reversed, and let $L^\vee$ be the distributive lattice whose poset of  join-irreducible  elements is $P^\vee$. Then it is easily seen that $L$ is level if and only if $L^\vee$ is level. Therefore it follows from Miyazaki's theorem (as remarked by him in his paper), that $L$ is also level if for all $x\in P$, all chains in $\hat{P}$ descending from $x$ have the same length. Thus we call a finite  poset $P$  a {\em Miyazaki poset} if for all $x\in P$ all chains in $\hat{P}$ ascending from $x$ have the same length or all chains in $\hat{P}$ descending from $x$ have the same length. Unfortunately $L$ may be level, though $P$ is not a Miyazaki poset, and this may happen even for regular planar  lattices. On the other hand, in Theorem~\ref{alsoviviana} it is shown that if $L$ is an arbitrary finite distributive lattice which is level, then for all $x,y\in P$ with $x\gtrdot y$ we must have that $\height(x)+\depth(y)  \leq \rank \hat{P}+1$. Here we use the notation  ``$x\gtrdot y$" to express that $x$ covers $y$, that is, $x>y$ and for any $z\in P$ either $z>x$ or $z<y$. At present we do not know whether these inequalities for all covering pairs in $P$ actually characterize the levelness of $L$. On other hand, it is shown in Theorem~\ref{hibi} that a regular planar lattice $L$ is level if and only if these inequalities hold for all covering pairs in $P$. We apply this result  to give in Theorem~\ref{butterfly} an explicit description of those distributive lattices $L$ whose poset $P$ of join-irreducible elements has a special shape which we call a butterfly. In this particular case it turns out that $L$ is level if and only if the initial ideal $\ini(I_L)$ defines a level ring where $I_L$ is the defining ideal of the Hibi ring $K[L]$.  One may wonder whether the regularity condition in Theorem~\ref{viviana} is really needed. In the case of a planar lattice with only one inside corner  the regularity hypothesis may indeed be dropped, as shown in Theorem~\ref{diagonalposet}.

In the last section of this paper we study the pseudo-Gorenstein and level property  of generalized Hibi rings. For a fixed field $K$,  a poset $P$ and any
integer $r$ one defines the so-called  generalized Hibi ring $\MR_r(P)$  which is naturally attached  to $r$-multichains of poset ideals in $P$, see \cite{EHM}.
For $r=2$ one  obtains the ordinary Hibi rings. In Theorem~\ref{new} it is shown that  $R_2(P)$ is pseudo-Gorenstein if and only if   $R_r(P)$ is
pseudo-Gorenstein for  some $r\geq 2$, and that $R_2(P)$ is level if $R_r(P)$ is level for some $r\geq 2$.

\section{Pseudo-Gorenstein rings}
\label{1}
Let $K$ be a field and  $R$  a Cohen-Macaulay standard graded $K$-algebra of dimension $d$ with canonical module  $\omega_R$. We choose a  presentation $R\iso S/I$ where $S=K[x_1,\ldots,x_n]$ is a polynomial ring and $I\subset \mm^2$ with $\mm=(x_1,\ldots,x_n)$. Furthermore, let $\FF$ be the graded minimal free resolution of $S/I$.
It is a simple exercise to see that the following conditions are equivalent:

\begin{enumerate}
\item[(i)] Let $a=\min\{i\:\; (\omega_R)_i\neq 0\}$. Then $\dim_K(\omega_R)_a=1$;
\item[(ii)] Let $y_1,\ldots,y_d$ be a maximal regular sequence of linear forms in $R$ and set $\bar{R}=R/(y_1,\ldots,y_d)R$. Furthermore, let  $b=\max\{i\:\; \bar{R}_i\neq 0\}$. Then
$\dim_K\bar{R}_b=1$.
\item[(iii)] Let $H_R(t)=P(t)/(1-t)^d$  be the Hilbert series of $R$. Then the leading coefficient of $P(t)$ is equal to $1$.
\item[(iv)] The highest shift $c$ in the resolution $\FF$ appears in $F_{n-d}$ and
\[
\beta_{n-d, c}(S/I)=1.
\]
\end{enumerate}

We call $R$ {\em pseudo-Gorenstein} if one (or all) of the above equivalent conditions hold. The ring $R$ is called {\em level} if $\omega_R$ is generated in a single degree. It is clear from (i) that a pseudo-Gorenstein ring is level if and only if it is Gorenstein.

Let $<$ be a monomial order and assume that $S/\ini_<(I)$ is Cohen-Macaulay. It follows from (iv) that  $S/\ini_<(I)$ is pseudo-Gorenstein if and only if $S/I$ is pseudo-Gorenstein. In particular, if $I\subset S$ is a toric ideal such that $S/I$ is Gorenstein and $\ini_<(I)$ is a squarefree monomial ideal, then $S/\ini_<(I)$ is pseudo-Gorenstein. Here we use that $S/\ini_<(I)$ is Cohen--Macaulay, if $I$ is a toric ideal and $\ini_<(I)$ is squarefree, see \cite[Corollary~6.6.18]{V}.

\section{The canonical module of a Hibi ring}
\label{pseudo}

Let $(L,\wedge,\vee)$ be a finite distributive lattice. An element  $\alpha\in L$ is called {\em join-irreducible} if $\alpha\neq \min L$ and  whenever $\alpha=\beta\vee \gamma$, then $\alpha=\beta$ or $\alpha=\gamma$. Let $P$ be the subposet of join-irreducible elements of $L$. By a well-known theorem of Birkhoff \cite{B}, one has that  $L\iso \MI(P)$, where $\MI(P)$ is the lattice of poset ideals of $P$ with the partial order given by inclusion and with union and intersection as join and meet operation. Poset ideals of $P$ are subsets $\alpha$ of $P$ with the property that if $x\in \alpha$ and $y\leq x$, then $y\in \alpha$. In particular, $\emptyset$ is a poset ideal of $P$.

Given a field $K$. The Hibi ring of $L$ over $K$ is the $K$-algebra $K[L]$ generated by the elements $\alpha\in L$ and with the defining relations $\alpha\beta-(\alpha\wedge \beta)(\alpha\vee \beta)$. Identifying $L$ with $\MI(P)$, it is shown in \cite{H} that $K[L]$ is isomorphic to the toric ring generated over $K$ by the elements  $u_\alpha$ with $\alpha \in L$,  where $u_\alpha$ is the monomial  $s\prod_{x\in \alpha}t_x$  in the polynomial ring $K[s,t_x:x\in P]$.

Let $\hat{P}$ be the poset $P\union \{\infty,-\infty\}$ with $-\infty <x<\infty$ for all $x\in P$. A map $v\: \hat{P}\to \ZZ_{\geq 0}$ is called {\em order reversing} if $v(x)\leq v(y)$ for all $x,y\in\hat{P}$  with $x\geq y$, and $v$ is a called {\em strictly order reversing} if  $v(x)<v(y)$ for all $x,y\in\hat{P}$  with $x>y$.

We denote by $\MS(\hat{P})$ the set of all order reversing functions $v$ on $\hat{P}$ with $v(\infty)=0$, and by $\MT(\hat{P})$ the set of all strictly order reversing functions $v$ on $\hat{P}$ with $v(\infty)=0$.

It is shown in \cite{H} that the toric ring $K[L]$ has a $K$-basis consisting of the monomials
\begin{eqnarray}
\label{reversing}
s^{v(-\infty)}\prod_{x\in P}t_x^{v(x)}, \quad v\in \MS(\hat{P}),
\end{eqnarray}
and that the monomials
\begin{eqnarray}
\label{strict}
s^{v(-\infty)}\prod_{x\in P}t_x^{v(x)}, \quad v\in \MT(\hat{P})
\end{eqnarray}
form a $K$-basis of the canonical ideal $\omega_L\subset K[L]$. The (finite) set of elements $v\in  \MT(\hat{P})$  which correspond to a minimal set of generators of $\omega_L$ will be denoted by $\MT_0(\hat{P})$.

It follows from this description of $\omega_L$ that the property of $K[L]$ to be pseudo-Gorenstein, Gorenstein or level does not depend on $K$. Thus we call $L$ itself pseudo-Gorenstein, Gorenstein or level if $K[L]$ has this property

\medskip
Note that $K[L]$ is standard graded with $\deg s^{v(-\infty)}\prod_{x\in P}t_x^{v(x)}=v(-\infty)$ for each of the monomials in (\ref{reversing}).

\medskip
Before proceeding we recall some basic concepts and notation regarding finite posets.
Let $Q$ be an arbitrary poset. A nonempty subposet $C$ of $P$ which is totally ordered is called a {\em chain} in $P$. The {\em length} of $C$ is defined to be $|C|-1$, and denoted $\ell(C)$. The rank of $Q$, denoted $\rank Q$, is defined to be the maximal length of a chain in $Q$. Let  $x\in Q$. Then  $\height_Q(x)$  (resp. $\depth_Q(x)$)  is defined to be  the maximal length of a chain descending  (resp.\ ascending) from $x$ in $Q$. In the case that $Q=\hat{P}$ for some poset $P$, we omit the lower index and simply write $\height(x)$ and $\depth(x)$.

Let $x,y \in P$. It is said that $x$ {\em covers} $y$, denoted $x\gtrdot y$, if $x>y$ and there exists no $z\in P$ such that $x>z>y$.

Let as before $L$ be a finite distributive lattice. Then $L$ is called {\em simple} if  there exist no elements $\alpha,\beta\in L$ with the property
$\alpha\gtrdot \beta$ and such that for each $\gamma\in L$ with $\gamma \neq\alpha,\beta$, we have $\gamma>\alpha$ or $\gamma<\beta$. Let $P$ be the poset of join-irreducible elements of $L$. Then $L$ is simple if and only if there exists no element $x\in P$ which is comparable with all elements in $P$.

\medskip

It is observed in \cite{EHS} that $\min\{v(-\infty)\:\; v \in \MT(\hat{P})\}=\rank \hat{P}$.
Thus it follows that $L$ is pseudo-Gorenstein if and only if there exists precisely one $v\in  \MT(\hat{P})$ with $v(-\infty)=\rank \hat{P}$, and that $L$ is level if and only if for any $v\in  \MT(\hat{P})$ there exists $v'\in  \MT(\hat{P})$ with $v'(-\infty)=\rank \hat{P}$ and such that $v-v'\in \MS(\hat{P})$.

\medskip
In the following  result we characterize  pseudo-Gorenstein distributive lattices in terms of their poset of join-irreducible elements.

\begin{Theorem}
\label{classification}
The distributive lattice   $L$ is pseudo-Gorenstein if and only if
\[\depth(x)+\height(x)=\rank \hat{P}\quad \text{for all}\quad  x\in P.
\]
\end{Theorem}

\begin{proof}
Suppose first that $\depth(x)+\height(x)=\rank \hat{P}$ for all $x\in P$. This implies that for any  $x\in \hat{P}$ there exists a chain $C$ of length equal to $\rank \hat{P}$ with $x\in C$. Now let $v$ be any strictly order reversing function on $\hat{P}$ with $v(\infty)=0$ and $v(-\infty)=\rank \hat{P}$. Then for any $y\in C$ we must have $v(y)=\depth(y)$. In particular, $v(x)=\depth(x)$. This shows that $v$ is uniquely determined, and proves that $L$ is pseudo-Gorenstein.

Conversely, suppose that $L$ is pseudo-Gorenstein. For all $x\in\hat{P}$ we set $v(x)=\depth(x)$ and $v'(x)=\rank \hat{P}-\height(x)$. Then both, $v$ and $v'$,  are strictly order reversing  functions on $\hat{P}$ with $v(\infty)=v'(\infty)=0$ and  $v(-\infty)=v'(-\infty)=\rank \hat{P}$. Since $L$ is pseudo-Gorenstein we have $v=v'$. This implies that  $\depth(x)+\height(x)=\rank \hat{P}$ for all $x\in P$.
\end{proof}

\section{Hyper-planar lattices}
\label{hyper}

Let $L$ be a finite distributive lattice and $P$ its poset of join-irreducible elements. We call $L$ a {\em hyper-planar lattice}, if $P$ as a set is the disjoint union of chains $C_1,\ldots, C_d$,  where each $C_i$ is a maximal chain in $P$.  We call such a chain decomposition {\em canonical}.  Of course in general an element $x\in C_i$ may be comparable with an element $y\in C_j$ for some  $j\neq i$. If this is the case and if  $x\gtrdot y$, then we call the chain $x\gtrdot y$ (of length one) a {\em diagonal} of $P$ (with respect  to the given canonical chain decomposition). For example, the poset depicted in Figure~\ref{counter} has two diagonals. If $d=2$, we recover the simple planar lattices.

A canonical chain decomposition of the poset $P$ of join-irreducible elements for a hyper-planar lattice $L$ is in general not  uniquely determined. However we claim  that if $C_1\union C_2\union\cdots \union C_s$ and $D_1\union D_2\union \cdots \union D_t$ are canonical chain decompositions of $P$, then $s=t$.

Indeed, let  $\max(Q)$ denote the set of maximal elements of a finite poset $Q$. Then
\begin{eqnarray}
\label{max}
\max(P)&=&\max(C_1)\union \max(C_2)\union\cdots \union \max(C_s)\\
&=&\max(D_1)\union \max(D_2)\union\cdots \union \max(D_t).\nonumber
\end{eqnarray}
Let $\max(C_i)=\{x_i\}$ for $i=1,\ldots,s$ and $\max(D_i)=\{y_i\}$ for $i=1,\ldots,t$. Then the elements $x_i$ as well as the elements $y_i$ are pairwise distinct, and  it follows from (\ref{max}) that
\[
\{x_1,x_2,\ldots,x_s\}=\{y_1,y_2,\ldots,y_t\},
\]
Hence we see that $s=t$.

\medskip
One would even  expect that the
\begin{eqnarray}
\label{equal}
\{\ell(C_1),\ell(C_2),\ldots, \ell(C_s)\} =\{\ell(D_1),\ell(D_2),\ldots,\ell(D_t)\},
\end{eqnarray}
as multisets. This however is not the case. For the poset $P$ displayed in  Figure~\ref{different} we have the following two canonical chain decompositions
\[
C_1=a<b<c<d<e<f, \quad C_2=g<h<i<j<k<l,
\]
and
\[
D_1=a<b<i<e<f, \quad D_2=g<h<c<d<j<k<l.
\]

Thus we see that  $\ell(C_1)=\ell(C_2)=5$, while  $\ell(D_1)=4$ and $\ell(D_2)=6$.

Also, note that the chain $i\gtrdot b$ is a diagonal of $P$, depicted in Figure~\ref{different}, with respect to the canonical chain decomposition $C_1\cup C_2$, but not is not a diagonal of $P$ with respect to  the canonical chain decomposition $D_1\cup D_2$.

\begin{figure}[hbt]
\begin{center}
\psset{unit=0.8cm}
\begin{pspicture}(-5,0)(4,6)

\rput(-1.5,0){
\psline(0,0)(0,5)
\rput(0,0){$\bullet$}
\rput(0,1){$\bullet$}
\rput(0,2){$\bullet$}
\rput(0,3){$\bullet$}
\rput(0,4){$\bullet$}
\rput(0,5){$\bullet$}
\psline(2,0)(2,6)
\rput(2,0){$\bullet$}
\rput(2,1){$\bullet$}
\rput(2,2.5){$\bullet$}
\rput(2,4){$\bullet$}
\rput(2,5){$\bullet$}
\rput(2,6){$\bullet$}
\psline(0,1)(2,2.5)
\psline(0,4)(2,2.5)
\psline(0,3)(2,4)
\psline(0,2)(2,1)
\rput(-0.3,0){$a$}
\rput(-0.3,1){$b$}
\rput(-0.3,2){$c$}
\rput(-0.3,3){$d$}
\rput(-0.3,4){$e$}
\rput(-0.3,5){$f$}
\rput(2.3,0){$g$}
\rput(2.3,1){$h$}
\rput(2.3,2.5){$i$}
\rput(2.3,4){$j$}
\rput(2.3,5){$k$}
\rput(2.3,6){$l$}
}
\end{pspicture}
\end{center}
\caption{}\label{different}
\end{figure}

In order to guarantee that  equality (\ref{equal}) is  satisfied we have to add an extra condition on the hyper-planar lattice:
let $L$ be a hyper-planar lattice whose poset of join-irreducible elements is  $P$. In what follows this will be our standard assumption and notation.

We say that $L$ is a {\em regular hyper-planar lattice}, if for any canonical chain decomposition $C_1\cup C_2\cup\ldots\cup C_d$ of $P$, and for all $x<y$ with $x\in C_i$ and $y\in C_j$ it follows that $\height_{C_i}(x)<\height_{C_j}(y)$.

\medskip

\begin{Lemma}
\label{regular}
Let $L$ be a regular hyper-planar lattice and $C_1\cup \ldots \cup C_d$ be a canonical chain decomposition of $P$. Then for all $i$ and $x\in C_i$ we have $\height_{C_i}(x)=\height_P(x)$.
\end{Lemma}

\begin{proof}
We proceed by induction on $\height_P(x)$. If $\height_P(x)=0$, then there is nothing to show. Now assume that $\height_P(x)>0$ and let $y\in P$ covered by $x$ with $\height_P(y)=\height_P(x)-1$. Say, $y\in C_j$. Since $\height_P(y)=\height_P(x)-1$ we may apply our induction hypothesis and obtain
\[
\height_P(x)-1=\height_P(y)=\height_{C_j}(y)<\height_{C_i}(x)\leq \height_P(x).
\]
This yields the desired conclusion.
\end{proof}

\begin{Corollary}
\label{ell}
Let $L$ be a regular hyper-planar lattice, and  assume that besides  of $C_1\cup \ldots \cup C_d$ there is still another canonical chain decomposition   $D_1\union D_2\union \cdots \union D_d$  of $P$. Then
\begin{enumerate}
\item[{\em (a)}] $\{\ell(C_1),\ell(C_2),\ldots, \ell(C_d)\} =\{\ell(D_1),\ell(D_2),\ldots,\ell(D_d)\}$,
as multisets.
\item[{\em (b)}]
$\rank P=\max\{\ell(C_1),\ldots,\ell(C_d)\}$.
\item[{\em (c)}]
$\height (x)+\depth (x)=\rank \hat{P}$ for all $C_i$ with $\ell(C_i)=\rank P$ and all $x\in C_i$.
\end{enumerate}
\end{Corollary}

\begin{proof}
Let $\max(C_i)=\{x_i\}$ and $\max(D_i)=\{y_i\}$ for $i=1,\ldots,t$. We have seen in the discussion before Lemma~\ref{regular} that
\[
\{x_1,x_2,\ldots,x_d\}=\{y_1,y_2,\ldots,y_d\},
\]
Therefore, \[\{\height_P(x_1),\height_P(x_2),\ldots,\height_P(x_d)\}=\{\height_P(y_1),\height_P(y_2),\ldots,\height_P(y_d)\},\] as multi-sets. By
Lemma~\ref{regular}, $\height_{P}(x_i)=\ell(C_i)$ and  $\height_{P}(y_i)=\ell(D_i)$. This together with the observation that $\rank P=\max\{\height_P(x_1),\height_P(x_2),\ldots,\height_P(x_d)\}$ proves (a) and (b).

In order to prove (c) we observe that
\begin{eqnarray*}
\rank \hat{P}&=&\ell(\hat{C_i})=\height_{\hat{C_i}}(x)+\depth_{\hat{C_i}}(x)\\
&\leq & \height (x)+\depth (x)\leq \rank \hat{P}.
\end{eqnarray*}
\end{proof}

Now we are able to characterize the regular hyper-planar lattices which are  pseudo-Gorenstein.

\begin{Theorem}
\label{equallength}
Let $L$ be a regular hyper-planar lattice and $C_1\cup \ldots \cup C_d$ be a canonical chain decomposition of $P$. Then $L$ is pseudo-Gorenstein if and only if all $C_i$ have the same length.
\end{Theorem}

\begin{proof}
Suppose all $C_i$ have the same length. Then Corollary~\ref{ell} implies that $\ell(C_i)=\rank \hat{P}$ for all $i$. Let $x\in P$. Then $x\in C_i$ for some $i$, and hence $\height (x)+\depth (x)=\rank \hat{P}$, by Corollary~\ref{ell}. Therefore, by Theorem~\ref{classification}, $L$ is pseudo-Gorenstein.

Conversely, suppose that not all $C_i$ have the same length. Then Corollary~\ref{ell} implies that there exists  one $C_i$  with $\ell(C_i)<\rank P$. As in the proof of Theorem~\ref{classification} we consider the strictly order reversing function $v(x)=\depth (x)$ and $v'(x)=\rank \hat{P}-\height (x)$. Let $x=\max(C_i)$. Then $v(x)=1$ and, since $L$ is regular, $v'(x)=\rank \hat{P}-(\ell(C_i)+1)>\rank \hat{P}-\rank P-1=1$. This shows that $L$ is not pseudo-Gorenstein.
\end{proof}

In \cite{M} Miyazaki showed that $L$ is level, if for all $x\in P$  all maximal  chains ascending from $p$ have the same length, or all  maximal  chains descending from $x$ have the same length. We call a poset $P$ with this property a {\em Miyazaki poset}.

\begin{Corollary}
\label{easy}
Let $L$ be a regular hyper-planar lattice and let $C_1\union C_2\union \cdots\union C_d$ be a canonical  chain decomposition of $P$. We   assume  that all $C_i$ have the same length. Then the following conditions are equivalent:
\begin{enumerate}
\item[{\em (a)}] $L$ is Gorenstein;
\item[{\em (b)}] $L$ is level;
\item[{\em (c)}] $P$ is a Miyazaki poset.
\end{enumerate}
\end{Corollary}

\begin{proof}
By Theorem~\ref{equallength}, $L$ is pseudo-Gorenstein. Thus (a) and  (b) are equivalent, as noticed in Section 1. The implication (c)\implies (b) follows by Miyazaki \cite{M}, and (a)\implies (c) follows by Hibi's theorem \cite{H} which says that $L$ is Gorenstein if and only if $P$ is pure.
\end{proof}

For simple planar lattices  Theorem~\ref{equallength} can be improved as follows.

\begin{Theorem}
\label{viviana}
Let $L$ be a simple planar lattice,  and let $C_1\union C_2$ be a canonical chain decomposition of $P$.
Then the following conditions are equivalent:
\begin{enumerate}
\item[{\em (a)}] $L$ is pseudo-Gorenstein;
\item[{\em (b)}] $L$ is regular and the chains $C_1$ and $C_2$ have the same length.
\end{enumerate}
\end{Theorem}

\begin{proof}
It suffices to prove (a)\implies (b) because the implication (b) \implies (a) is a special case of Theorem~\ref{equallength}. Now for the proof of the implication (a)\implies (b) we use the characterization of pseudo-Gorenstein lattices via the Hilbert series of $K[L]$, as given in (iii), Section 1. Namely,  if $H_{K[L]}(t)=P(t)/(1-t)^d$ with $d=\dim K[L]$, then $L$ is pseudo-Gorenstein if and only if the leading coefficient of $P(t)$ is one. Note that $L$ may be identified with a two-sided ladder inside an $m\times n$ rectangle with $n\leq m$ as shown in Figure~\ref{ladder}.
\texttt{}\begin{figure}[hbt]
\begin{center}
\psset{unit=1cm}
\begin{pspicture}(-5,0)(4,6)

\rput(-4,0){
\rput(-0.3,-0.3){$(0,0)$}
\psline(0,0)(2,0)
\psline(2,0)(2,1)
\psline(2,1)(3,1)
\psline(3,1)(3,2)
\psline(3,2)(6,2)
\psline(6,2)(6,3)
\psline(6,3)(7,3)
\psline(7,3)(7,5)
\psline[linestyle=dotted](2,0)(7,0)
\psline[linestyle=dotted](7,0)(7,3)

\rput(3.5,3){$L$}

\rput(7.3,5.3){$(m,n)$}
\psline(0,0)(0,3)
\psline(0,3)(2,3)
\psline(2,3)(2,4)
\psline(2,4)(3,4)
\psline(3,4)(3,5)
\psline(3,5)(7,5)
\psline[linestyle=dotted](0,3)(0,5)
\psline[linestyle=dotted](0,5)(3,5)

}
\end{pspicture}
\end{center}
\caption{}\label{ladder}
\end{figure}
Now according to \cite{ERQ} the leading coefficient of $P(t)$ is the number of maximal cyclic sublattices of $L$. Thus $L$ is pseudo-Gorenstein if and only if $L$ admits precisely one maximal cyclic sublattice. By a cyclic sublattice of $L$, we mean a sublattice inside the two-sided ladder  in Figure~\ref{ladder} consisting of a chain of  squares and edges as in the example shown in Figure~\ref{cyclic}. We call the cyclic sublattice maximal if it has the maximal number of squares among all cyclic sublattices contained in $L$.

\texttt{}\begin{figure}[hbt]
\begin{center}
\psset{unit=1cm}
\begin{pspicture}(-5,0)(4,6)

\rput(-3,0){
\psline(0,0)(0,1)
\psline(0,0)(1,0)
\psline(0,1)(1,1)
\psline(1,0)(1,1)
\rput(0,0){$\bullet$}
\rput(0,1){$\bullet$}
\rput(1,0){$\bullet$}
\rput(1,1){$\bullet$}

\psline(1,1)(1,2)
\psline(1,2)(1,3)
\rput(1,2){$\bullet$}
\rput(1,3){$\bullet$}

\psline(1,3)(2,3)
\psline(2,3)(2,4)
\psline(2,4)(1,4)
\psline(1,3)(1,4)
\rput(2,3){$\bullet$}
\rput(2,4){$\bullet$}
\rput(1,4){$\bullet$}

\psline(2,4)(3,4)
\rput(3,4){$\bullet$}

\psline(3,4)(4,4)
\psline(4,4)(4,5)
\psline(4,5)(3,5)
\psline(3,5)(3,4)
\rput(3,5){$\bullet$}
\rput(4,4){$\bullet$}
\rput(4,5){$\bullet$}

\psline(4,5)(5,5)
\psline(5,5)(5,6)
\psline(5,6)(4,6)
\psline(4,5)(4,6)
\rput(4,6){$\bullet$}
\rput(5,6){$\bullet$}
\rput(5,5){$\bullet$}
}
\end{pspicture}
\end{center}
\caption{}\label{cyclic}
\end{figure}

In the first step of our proof we show by induction on $n+m$ that if $L$ is pseudo-Gorenstein, then $m=n$ and the lower inside corners of $L$ are below the diagonal connecting $(0,0)$ with $(n,n)$, while the upper inside corners are above this diagonal. In other words, if $(i,j)$ is a lower inside corner of $L$, then $i\geq j$, while for an upper inside corner $(i,j)$ we have $i\leq j$. If this is the case, then we say that the inside corners do not cross the diagonal.

If $m+n=2$, then there is nothing to prove. Let $L'$ be the maximal sublattice of $L$ (again viewed as a ladder) with the bottom and top elements $(0,0)$ and $(m-1,n-1)$, respectively.  Now, we consider two cases. Suppose first that the integral points of the square   $[(m-2,n-2),(m-1,n-1)]$ belong to $L'$. In this case, $L'$ is simple since $L$ is simple. We claim that $L'$ is pseudo-Gorenstein. Indeed, suppose this is not the case. Then $L'$  has at least two different maximal cyclic sublattices. Then because $L$ is simple, the square  $[(m-1,n-1),(m,n)]$ is a subset of $L$. Therefore, each of these maximal sublattices may be extended in $L$ with the square  $[(m-1,n-1),(m,n)]$, contradicting our assumption that $L$ is pseudo-Gorenstein.   Thus, the induction hypothesis implies that  $m-1=n-1$, so that $m=n$. Also, the induction hypothesis implies that the inside corners of $L'$ do not cross the diagonal, and hence this is also the case for $L$.   Now, suppose that $L'$ does not contain the square $[(m-2,n-2),(m-1,n-1)]$.   Then we may assume that for some $j<n-1$ the ladder $L$ contains the squares $[(m-1,k),(m,k+1)]$ for all $k$ with $j\leq k\leq n-1$, but does not contain the squares $[(i,k),(i+1,k)]$ for all  $i$ and $j$ with $1\leq i\leq m-2$ and $j\leq k\leq n-1$. Then $K[L']$ has the same $h$-vector as $K[L'']$ where $L''$ is a sublattice of $L$ which viewed as a ladder is contained in $[(0,0),(m-1,j)]$. Note that $j\neq m-1$, because $j<n-1\leq m$. Therefore, since $L''$ is simple, our induction hypothesis implies that $L''$ is not pseudo-Gorenstein. Hence there exist at least two maximal cyclic sublattices in $L''$, and each of these cyclic sublattices may be extended to maximal cyclic sublattices in $L$, which contradicts the fact that $L$ is pseudo-Gorenstein. Thus this second case is  not possible.

Now we are ready to prove (b): Let $C_1=x_1<x_2<\cdots <x_m$  and $C_2=y_1<y_2<\cdots<y_n$. Then $L$ viewed as a two-sided ladder contains the points $(0,0)$ (corresponding to the poset ideal $\emptyset$ of $L$), and $(m,n)$ (corresponding to the poset ideal $L$). Since $L$ is pseudo-Gorenstein, it follows that $m=n$, as we have seen before. Being regular is equivalent to the condition that the inside corners of a ladder $L$ do not cross the diagonal connecting $(0,0)$ and $(n,n)$. In fact, the join-irreducible elements of $L$ establishing the chain $C_2$ can be identified with the vertices of the ladder (as displayed in Figure~\ref{ladder}) which are located on the vertical border lines of the upper border and are different from the inside corners and different from $(0,0)$, while the  join-irreducible elements of $L$ forming  the chain $C_1$ can be identified with the vertices of the ladder which are located on the horizontal  border lines of the lower  border of $L$ and are different from the inside corners and different from $(0,0)$. After this identification let $x=(i,j)\in C_2$ be  and $y=(k,l)\in C_1$, Then $\height_{C_2}(x)= j$ and $\height _{C_1}(y)=k$. Assume now that $x>y$. Then this implies that $i\geq k$. Since the inside corners of $L$ do not cross the diagonal we have $j>i$, and thus $\height_{C_2}(x)=j>i\geq k=\height _{C_1}(y)$. Similarly, one shows that  $\height_{C_2}(x)<\height_{C_1}(y)$, if $x<y$. This completes the proof.
\end{proof}

Theorem~\ref{viviana} is not valid if the hyper-planar lattice is not planar, as the example displayed in Figure~\ref{notvalid} demonstrates. Indeed,  the lattice $L$ corresponding to $P$ is pseudo-Gorenstein, but in this example we only have one canonical chain decomposition, and the chains of this decomposition have different lengths. Moreover,  $L$   is not regular.

\texttt{}\begin{figure}[hbt]
\begin{center}
\psset{unit=1cm}
\begin{pspicture}(-5,0)(4,2)

\rput(-1.7,0){
\psline(0,0)(0,2)
\rput(0,0){$\bullet$}
\rput(0,1){$\bullet$}
\rput(0,2){$\bullet$}

\psline(2,0)(2,2)
\rput(2,0){$\bullet$}
\rput(2,1){$\bullet$}
\rput(2,2){$\bullet$}

\psline(1,0)(1,2)
\rput(1,0){$\bullet$}
\rput(1,2){$\bullet$}

\psline(0,1)(1,0)
\psline(1,2)(2,1)
}
\end{pspicture}
\end{center}
\caption{}\label{notvalid}
\end{figure}

We also would like to remark that in Theorem~\ref{viviana}(b) the condition ``regular" is required.  Indeed the poset shown in Figure~\ref{counter}  is the poset of join-irreducible elements of a non-regular simple planar lattice $L$ for which $L$ is not pseudo-Gorenstein.

\begin{figure}[hbt]
\begin{center}
\psset{unit=0.8cm}
\begin{pspicture}(-5,0)(4,6)

\rput(-1.5,0){
\psline(0,0)(0,4)
\rput(0,0){$\bullet$}
\rput(0,1){$\bullet$}
\rput(0,2){$\bullet$}
\rput(0,3){$\bullet$}
\rput(0,4){$\bullet$}

\psline(2,0)(2,4)
\rput(2,0){$\bullet$}
\rput(2,1.5){$\bullet$}
\rput(2,2.5){$\bullet$}
\rput(2,3.25){$\bullet$}
\rput(2,4){$\bullet$}
\psline(0,1)(2,1.5)
\psline(0,3)(2,2.5)
}
\end{pspicture}
\end{center}
\caption{}\label{counter}
\end{figure}

\section{Level distributive lattices}
\label{level}

Throughout this section  $L$ will be a finite distributive lattice and $P$ its poset of join-irreducible elements.  In the previous section we recalled the fact that $L$ is level if $P$ is a Miyazaki poset. In his paper \cite{M} Miyazaki mentioned the fact that his condition on $P$ is only a sufficient condition. One may ask whether for hyper-planar lattices a stronger result is possible.

We begin with a necessary condition for levelness which is valid for any distributive lattice.

\begin{Theorem}
\label{alsoviviana}
Suppose   $L$ is level.  Then
\begin{eqnarray}
\label{inequality}
\height(x)+\depth(y)  \leq \rank \hat{P}+1
\end{eqnarray}
for all $x,y\in P$ with  $x\gtrdot y$.
\end{Theorem}

\begin{proof}
Let $x,y\in P$ such that $x$ covers $y$ and suppose that $\height(x)+\depth(y)  > \rank \hat{P}+1$. We have to show that $L$ is not level.

Our assumption implies that
\[
\height(x)+\depth(y)  > \rank \hat{P}+1\geq \height(x)+\depth(x)+1,
\]
and hence
\[
\depth(y)>\depth(x)+1.
\]
We show that there exists $w\in \MT_0(\hat{P})$ with $w(-\infty)>\rank \hat{P}$. This then proves that $L$ is not level.

Let $\depth(y)-\depth(x)-1=\alpha.$ Then $\alpha>0$. We define $v\: \hat{P}\to {\ZZ}_{\geq 0}$ as follows:

\[
v(z)= \left\{ \begin{array}{ll}
       \depth(z)+\alpha, & \;\textnormal{if $x\geq z$, $z\neq y$}, \\ \depth(z), & \;\text{otherwise.}
        \end{array} \right.
\]
Then $v\in \MT(\hat{P})$. If $v\in \MT_0(\hat{P})$, then we are done, since
\[
v(-\infty)= \depth(-\infty) +\alpha = \rank \hat{P} +\alpha\geq \rank \hat{P}+1.
\]
The last inequality follows from the fact that $\alpha>0$.

On the other hand, if $v\not\in \MT_0(\hat{P})$, then   there exists $w\in\MT_0(\hat{P})$ with $v-w\in \MS(\hat{P})$. It follows that
\[
0\leq v(x)-w(x)\leq v(y)-w(y)=\depth(y)-w(y)\leq 0.
\]
Hence
\[
w(x)=v(x)= \depth(x)+\alpha= \depth(x)+\depth(y)-\depth(x)-1=\depth(y)-1.
\]
Let
\[
x=z_0>z_1>\cdots > z_k=-\infty
\]
be a chain whose length is $\height(x)$. Then
\[
w(x)<w(z_1)<\cdots < w(z_k)=w(-\infty),
\]
which implies that
\[
w(-\infty)\geq w(x)+\height(x)=(\depth(y)-1)+\height(x) > \rank \hat{P}.
\]
\end{proof}

In his paper \cite{M} Miyazaki remarked  that  for all $z\in P$ all chains  ascending from $z$ have the same length if and only if  for all $x,y\in P$ with $x\gtrdot y$, we have  $\depth(y) = \depth(x)+1$.
Therefore,  $P$ is a Miyazaki poset  if and only if
\[
\text{ $\depth(y) = \depth(x)+1$ for all $x,y\in P$ with $x\gtrdot y$,}
\]
or
\[
\text{$\height(x) = \height(y)+1$ for all $x,y\in P$ with $x\gtrdot y$.}
\]
In either case  $L$ is level.

\begin{Corollary}
Suppose $L$ is pseudo-Gorenstein and  $P$ satisfies the inequality (\ref{inequality}) for all $x,y\in P$ with $x\gtrdot y$.   Then  $\depth(y) = \depth(x)+1$ and $\height(x) = \height(y)+1$ for all $x,y\in P$ with $x\gtrdot y$. In particular, $L$ is level and hence Gorenstein.
\end{Corollary}

\begin{proof}
For all $x,y\in P$ with $x\gtrdot y$, we have $\height (x)+\depth(x) = \rank \hat{P}$, since $L$ is pseudo-Gorenstein. Thus, by the inequality (\ref{inequality}), we have $\height (x)+\depth(y)\leq \height (x)+\depth(x)+1$, and hence $\depth(y)\leq \depth(x)+1$. On the other hand, clearly, we have $\depth(y)\geq \depth(x)+1$, which implies the first desired formula. The formula regarding height is similarly obtained. Therefore, $L$ is level, because $P$ is a Miyazaki poset.
\end{proof}

As mentioned before, if $P$ is not a Miyazaki poset, then $L$ may nevertheless be level, and this may happen even if $L$  is a regular  simple planar lattices.
Figure~\ref{butterflyposet} shows a  poset which is not a Miyazaki poset.  However its  ideal lattice is a   regular   simple planar lattice and is level.

\begin{figure}[hbt]
\begin{center}
\psset{unit=0.8cm}
\begin{pspicture}(-5,2)(4,6)

\rput(-1.5,2){
\psline(0,0)(0,1.5)
\rput(0,0){$\bullet$}
\rput(0,1.5){$\bullet$}

\psline(2,0)(2,2)
\rput(2,0){$\bullet$}
\rput(2,1){$\bullet$}
\rput(2,2){$\bullet$}
\psline(0,0)(2,2)
\psline(0,1.5)(2,0)
}
\end{pspicture}
\end{center}
\caption{}\label{butterflyposet}
\end{figure}

The following result shows that for regular planar lattices the necessary condition for levelness formulated in Theorem~\ref{alsoviviana} is also sufficient.

\begin{Theorem}
\label{hibi}
Let  $L$ be  a regular planar  lattice.  Then the following conditions are equivalent:
\begin{enumerate}
\item[{\em (a)}] $L$ is level;
\item[{\em (b)}] $\height(x)+\depth(y)  \leq \rank \hat{P}+1$ for all $x,y\in P$ with $x\gtrdot y$;
\item[{\em (c)}]  for all $x,y\in P$ with $x\gtrdot y$, either  $\depth(y) = \depth(x)+1$  or $\height(x) = \height(y)+1$.
\end{enumerate}
\end{Theorem}

\begin{Remark}
\label{check diagonal}
{\em (i) Observe  that a Miyazaki poset satisfies condition (c). On the other hand,  Figure~\ref{butterflyposet} shows a poset satisfying condition (c) which is not Miyazaki.

(ii) Let $C_1\cup C_2$ be a canonical chain decomposition of $P$. The inequality in (b) and the equations in (c) are always satisfied for those $x\gtrdot y$ for which $x$ and $y$ belong to the same chain in the decomposition. Hence it suffices to check the inequality in (b) and equations in (c) only for diagonals. Indeed, this fact follows directly from  Lemma~\ref{regular}. For instance, if $x$ covers $y$ and both  belong to the same chain, then, by Lemma~\ref{regular}, $\height (x)=\height(y)+1$. Thus $\height(x)+\depth(y)=\height(y)+1+\depth(y)\leq \rank \hat{P} +1$. For (c) we only need to observe that the condition $\height (x)=\height(y)+1$ for any $x$ that  covers $y$ in the same chain, is always fulfilled, again by Lemma~\ref{regular}. }
\end{Remark}

\medskip
Before proving Theorem~\ref{hibi} we  will need the following result.

\begin{Lemma}
\label{biggerone}
Let  $L$ be  a regular planar lattice. Let $C_1\union C_2$ be a canonical chain decomposition of $P$, and  assume  that $\ell(C_1)=\rank P$ (cf.\ Corollary~\ref{ell}). Suppose that $P$ satisfies condition (b) of Theorem~\ref{hibi}. Then for every $v\in \MT_0(\hat{P})$ we have $v(\max (C_1))=1$.
\end{Lemma}

\begin{proof}
Assume that $v(\max(C_1))>1$. Then $v(z)\geq \depth(z)+1$ for all $z\in C_1$.

Let
\[
v'(x)= \left\{ \begin{array}{ll}
       v(x)-1, & \;\textnormal{if $v(x)\geq \depth(x)+1$\;  (I),} \\ v(x), & \;\text{if $v(x)=\depth(x)$\hspace{0.8cm} (II),}
        \end{array} \right.
\]
for all $x\in \hat{P}$.

We show that $v'\in \MT(\hat{P})$ and  $v-v'\in \MS(\hat{P})$. Since $v'\neq v$, this will then show that $v\not\in \MT_0(\hat{P})$, a contradiction. Indeed, to see that  $v'\in \MT(\hat{P})$ we have to show that $v'(x)<v'(y)$ for all $x\gtrdot y$. If both $x$ and $y$ satisfy (I) or (II), then the assertion is trivial. If $x$ satisfies (I) and $y$ satisfies (II), then $v'(x)=v(x)-1<v(y)=v'(y)$, and if  $x$ satisfies (II) and $y$ satisfies (I), then $v(x)=\depth(x)\leq \depth(y)-1\leq v(y)-2$. Hence $v(x)<v(y)-1$, and this implies that $v'(x)<v'(y)$.

It remains to be shown that $v-v'\in \MS(\hat{P})$ which amounts to prove that $v(x)-v'(x)\leq v(y)-v'(y)$ for all $x\gtrdot y$. For this we only need to show that we cannot have $v'(x)=v(x)-1$ and $v(y)=v'(y)$, or, equivalently, that  $v(x)\geq \depth(x)+1$ and $v(y)=\depth(y)$ is impossible.

Assume to the contrary that there exist $x\gtrdot y$ with $v(x)\geq \depth(x)+1$ and $v(y)=\depth(y)$ . Then $y\not\in C_1$ since $v(z)\geq \depth(z)+1$ for all $z\in C_1$. Thus, we may either have $x\in C_1$ and $y\in C_2$, or $x,y\in C_2$.

In the first case, since  $\height(x)+\depth(y)\leq \rank \hat{P}+1$ by assumption, and since $\rank \hat{P}=\height(x)+\depth(x)$ due to the regularity of $L$  (see Corollary~\ref{ell}), we get $\depth(y)\leq \depth(x)+1$,  and hence $\depth(y)=\depth(x)+1$. Therefore,  $\depth(y)=\depth(x)+1\leq v(x)<v(y)$, a contradiction.

Finally, let $x,y\in C_2$. Since $v(x)<v(y)$, it follows that  $\depth(y)>\depth(x)+1$.  Therefore, the longest chain from $y$ to $\infty$ cannot pass through $x$. This implies that there exists $z\in C_1$ with $z\gtrdot y$. As in the first case we then deduce that $v(y)>\depth(y)$. So we get again a contradiction.
\end{proof}

\begin{proof}[Proof of Theorem~\ref{hibi}]
(a) \implies (b) follows from Theorem~\ref{alsoviviana}.

(b) \implies (c): Let $C_1\cup C_2$ be a canonical chain decomposition of $P$ with $|C_1|\geq |C_2|$.

If $x,y \in C_1$ or $x,y \in C_2$, then by Lemma~\ref{regular}, it follows that $\height(x) = \height(y)+1$.

Next suppose that $x\in C_1$. Since $L$ is regular, we may apply Corollary~\ref{ell} and conclude that $\height(x)+\depth(x)=\rank \hat{P}$. Thus, by (b), we get $\depth (y)\leq \depth (x)+1$. On the other hand, it is clear that $\depth (y)\geq \depth (x)+1$. So that $\depth (y)=\depth (x)+1$.
Finally, if $y\in C_1$, then by Corollary~\ref{ell} we have $\height(y)+\depth(y)=\rank \hat{P}$. As in the previous case, we conclude that $\height(x) = \height(y)+1$.

(c)\implies (b): If $\depth(y)=\depth(x)+1$, then $\height(x)+\depth(y)=\height(x)+\depth(x)+1\leq \rank \hat{P}+1$, and if $\height (x)=\height(y)+1$, then $\height(x)+\depth(y)=\height (y)+\depth(y)+1\leq \rank \hat{P}+1$.

(b) \implies (a): As in Lemma~\ref{biggerone}  we let  $C_1\union C_2$ be a canonical chain decomposition of $P$, and  may assume  that $\ell(C_1)=\rank P\geq \ell(C_2)$. Let $v\in \MT_0(\hat{P})$. We will  show that there exists $v'\in \MT(\hat{P})$ with $v'(-\infty)=\rank \hat{P}$ and such that $v-v'\in \MS(\hat{P})$. Since $v$ is a minimal generator it follows that $v=v'$, and we are done.

In order to construct $v'$ we consider the subposet $Q$ of $P$ which is obtained from $P$ by removing the maximal elements $\max(C_1)$ and $\max(C_2)$. We define on $\hat{Q}$ the strictly order reversing function $u$  by $u(\infty)=0$,  and $u(z)=v(z)-1$ for all other $z\in \hat{Q}$. We notice that the ideal lattice of $Q$ is again a regular planar  lattice satisfying (b). Indeed, assume that there exist $x\gtrdot y$ with $x,y\in Q$ such that $\height_{\hat{Q}}(x)+
\depth_{\hat{Q}}(y)>\rank \hat{Q}+1=\rank \hat{P}$. Since $\height_{\hat{Q}}(x)=\height(x)$ and $\depth(y)=\depth_{\hat{Q}}(y)+1$, it follows that
\[
\height(x)+\depth(y)=\height_{\hat{Q}}(x)+\depth_{\hat{Q}}(y)+1>\rank \hat{P}+1,
\]
a contradiction.

Therefore, by induction on the rank we may assume that the ideal lattice of $Q$ is level. Hence there exists $w\in \MT(\hat{Q})$ with $w(-\infty)=\rank \hat{Q}=\rank \hat{P}-1$ and such that $u-w\in \MS(\hat{Q})$. Set $v'(z) =1+w(z)$ for all $z\in A=Q\union\{-\infty\}$. Then $v'$ is a strictly order reversing function on $A$ with $v'(-\infty)=\rank \hat{P}$ and such that $v-v'$ is order reversing on $A$. It remains to define $v'(C_i)$ for $i=1,2$ in a way such  that $v'\in \MT(\hat{P})$ and $v-v'\in \MS(\hat{P})$. We have to set $v'(\max(C_1))=1$ since $v(\max(C_1))=1$, and of course $v'(\infty)=0$. Let $x=\max(C_2)$ and let $z\in C_2$ be the unique element with $x\gtrdot z$. We set $v'(x)= v(x)-u(z)+w(z)= v(x)-v(z)+1+w(z)$, and claim that this $v'$ has the desired properties. Indeed, $v'(x)=v(x)-(v(z)-1-w(z))\leq v(x)$ and $v'(x)<1+w(z)=v'(z)$, since $v(x)<v(z)$. If $z$ is the only element covered by $x$, we are done. Otherwise, there exists $y\in C_1$ with $x\gtrdot y$ and it remains to be shown that $v'(y)>v'(x)=v(x)-v(z)+1+w(z)$. Suppose we know that $\depth_{\hat{Q}}(y)\geq w(z)$,  then
\[
v'(y)=w(y)+1\geq \depth_{\hat{Q}}(y)+1>w(z)\geq v'(x),
\]
as desired, since $v(x)-v(z)+1\leq 0$.  Thus in order to complete the proof we have to show that $\depth_{\hat{Q}}(y)\geq w(z)$. Since the ideal lattice of $\hat{Q}$ is regular, this is equivalent to showing that
\begin{eqnarray}
\label{last}
w(z)\leq \rank \hat{Q}-\height_{\hat{Q}}(y).
\end{eqnarray}
The assumption  (b) and  Corollary~\ref{ell}(c)  imply that
\[
\height(x)+\depth(y)\leq \rank \hat{P}+1=\height(y)+\depth(y)+1,
\]
so that $\height(x)\leq \height(y)+1$. This yields
\begin{eqnarray}
\label{good}
\height(x)= \height(y)+1
\end{eqnarray}
since  $\height(x)\geq \height(y)+1$ is always valid.

On the other hand, since $L$ is regular, Lemma~\ref{regular} implies that  $\height_P(x)= \height_{C_2}(x) =\height_{C_2}(z)+1=\height_{P}(z)+1$. This implies that $\height (x)=\height (z)+1$. So together with (\ref{good})  we then conclude that $\height(y)=\height(z)$. Since $\height_{\hat{Q}}(y)= \height(y)$ and $\height(z)= \height_{\hat{Q}}(z)$, inequality (\ref{last}) becomes $w(z)\leq \rank \hat{Q}-\height_{\hat{Q}}(z)$, and since $w(-\infty) =\rank \hat{Q}$,  this inequality indeed holds. This completes the proof of the theorem.
\end{proof}

In the  following theorem we discuss an example of a poset for which  the conditions of  Theorem~\ref{alsoviviana} can be made more explicit. Let $P$ be a finite poset with a canonical chain decomposition $C_1\cup C_2$ with $2\leq |C_1|\leq |C_2|$. For $i=1,2$, let $x_i$ be the maximal and $y_i$ the minimal element of $C_i$. We call $P$ a {\em butterfly poset} (of type $(C_1,C_2)$), if $x_1\gtrdot y_2$ and $x_2\gtrdot y_1$ are the only diagonals of $P$. Figure~\ref{butterflyposet} displays a butterfly poset. Obviously, the ideal lattice of a butterfly poset is regular.

For the next result, we need some notation. The Hibi ring $K[L]$ can be presented as the quotient ring $T/I_L$, where $T$ is the polynomial ring over $K$ in the variables $x_{\alpha}$ with $\alpha\in L$ and where $I_L$ is generated by the binomials $x_{\alpha}x_{\beta}-x_{\alpha\vee \beta}x_{\alpha\wedge \beta}$. In the following theorem, we consider a monomial order $<$ given by a height reverse lexicographic monomial order, that is, the
reverse lexicographic monomial order induced by a total ordering of the variables satisfying $x_{\alpha} <
x_{\beta}$ if $\height_{L} (\alpha)> \height_{L} (\beta)$.

\begin{Theorem}
\label{butterfly}
Let  $P$ be  a butterfly poset of type $(C_1,C_2)$, and $L$ its ideal lattice.  Then the following conditions are equivalent:
\begin{enumerate}
\item[{\em (a)}] $T/\mathrm{in}_<(I_L)$ is level;
\item[{\em (b)}] $L$ is level;
\item[{\em (c)}] $\height(x)+\depth(y)  \leq \rank \hat{P}+1$ for all $x,y \in P$ with $x\gtrdot y$;
\item[{\em (d)}]  for all $x,y \in P$ with $x\gtrdot y$, either  $\depth(y) = \depth(x)+1$  or $\height(x) = \height(y)+1$;
\item[{\em (e)}] $|C_1|=2$.
\end{enumerate}
\end{Theorem}

\begin{proof}
(a)\implies (b) is well-known. Conditions (b), (c) and (d) are equivalent by Theorem~\ref{hibi}, since $L$ is regular.

(c)\implies (e): Since $2\leq |C_1|\leq |C_2|$, we have $\height (x_1)=|C_1|$, $\depth (y_2)=|C_2|$ and $\rank  \hat{P}=|C_2|+1$.
On the other hand, by condition (b), we have $\height (x_1)+\depth (y_2)\leq \rank \hat{P}+1$. So, $|C_2|\leq 2$, and hence $|C_2|=2$.

(e)\implies (a): It is shown in \cite{H} that $\ini_<(I_L)$ is generated by the  monomials $x_{\alpha}x_{\beta}$ where  $\alpha$ and $\beta$ are  incomparable elements of $L$. Thus, $\mathrm{in}_<(I_L)$ is the Stanley-Reisner ideal of the order complex $\Delta$ of $L$. It is known that $\Delta$ is pure shellable, and hence Cohen-Macaulay, see \cite[Theorem~6.1]{Bj}. Since $\dim (T/\ini_{<}(I_L))=\dim(T/I_L)=|P|+1$, it follows that $\pd (T/\ini_{<} (I_L))=\pd (T/I_L)=|L|-|P|-1$.
Moreover, since $T/\ini_{<}(I_L)$ and $T/I_L$ are Cohen-Macaulay, their regularity is given by the degree of the numerator polynomial of their Hilbert series. Hence, since both Hilbert series coincide, their regularity is the same and we obtain $\reg (T/\mathrm{in}_<(I_L))=\reg (T/I_L)=|P|-\rank P$, by \cite{EHS}. Since  $|C_1|=2$,  the lattice $L$ viewed as a ladder is of the form as shown in Figure~\ref{butterflylattice}. So, we see that  $|L|=3|C_2|+1$. Therefore, since $|P|=|C_2|+2$, we see that $\pd (T/\mathrm{in}_<(I_L))=2|C_2|-2$ and $\reg (T/\mathrm{in}_<(I_L))=2$.
\texttt{}\begin{figure}[hbt]
\begin{center}
\psset{unit=1cm}
\begin{pspicture}(-5,0)(4,3)

\rput(-4,0){
\rput(-0.4,-0.4){$(0,0)$}
\psline(0,0)(2,0)
\psline(3,0)(5,0)
\psline(6,0)(7,0)

\psline(0,0)(0,1)
\psline(1,0)(1,2)
\psline(2,0)(2,2)
\psline(3,0)(3,2)
\psline(4,0)(4,2)
\psline(5,0)(5,2)
\psline(7,0)(7,2)
\psline(6,0)(6,2)
\psline(8,1)(8,2)

\psline(0,1)(2,1)
\psline(3,1)(5,1)
\psline(6,1)(8,1)

\psline(1,2)(2,2)
\psline(3,2)(5,2)
\psline(6,2)(8,2)

\psline[linestyle=dotted](2,0)(3,0)
\psline[linestyle=dotted](5,0)(6,0)
\psline[linestyle=dotted](2,1)(3,1)
\psline[linestyle=dotted](5,1)(6,1)
\psline[linestyle=dotted](2,2)(3,2)
\psline[linestyle=dotted](5,2)(6,2)

\rput(8.4,2.4){$(|C_2|,2)$}
}
\end{pspicture}
\end{center}
\caption{}\label{butterflylattice}
\end{figure}

So to prove (a), it is enough to show that $\beta_{2|C_2|-2,j}(T/\mathrm{in}_<(I_L))=0$, for all $j< 2|C_2|$. By  Hochster's formula, we have
\[
\beta_{2|C_2|-2, j}(T/\mathrm{in}_<(I_L))=\sum_{W\subseteq V, |W|=j}\dim _{K}\widetilde{H}_{j-2|C_2|+1}({\Delta}_W;K),
\]
where $V$ is the  vertex set of  $\Delta$, and where ${\Delta}_W$ denotes as usual the  subcomplex of  $\Delta$ induced by $W$. For $W\subseteq V$ with
$|W|=j<2|C_2|-1$,  $\dim _{K}\widetilde{H}_{j-2|C_2|+1}({\Delta}_W;K)=0$, since $j-2|C_2|+1<0$. Let $W\subseteq V$ with $|W|=2|C_2|-1$. We show that
$\dim _{K}\widetilde{H}_{0}({\Delta}_W;K)=0$ or equivalently ${\Delta}_W$ is a connected simplicial complex. As the connectedness of ${\Delta}_W$ is equivalent to the connectedness of the 1-skeleton ${\Delta}_W^{(1)}$ of ${\Delta}_{W}$, we show that the graph $G={\Delta}_W^{(1)}$ is connected. The vertices of $G$ correspond to the lattice points of the elements in $W$ which is a certain subset of the lattice points of the ladder displayed in Figure~\ref{butterflylattice}. Two vertices of $G$ are adjacent if they are contained in a chain in the lattice $L$. In other words, two vertices $(i,j)$ and $(i',j')$ with $j\leq j'$ are adjacent in $G$ if $j=j'$ or $i\leq i'$.

Let $v=(i,j)$ and $w=(i',j')$ be two nonadjacent vertices in $G$, so that we may assume $j<j'$ and $i'<i$. We show that there exists a path in $G$ connecting  $v$ and $w$.

Case 1. There exists $(s,t)\in W$ such that either $s\leq i'$, $t\leq j$ or $s\geq i$, $t\geq j'$. Then $(s,t)$ is adjacent to both of $v$ and $w$, and hence $v$ and $w$ are connected.

Case 2. There exists $s$ with $i'<s<i$ such that $(s,j),(s,j')\in W$. Then we get a path in $G$ between $v$ and $w$ passing through $(s,j)$ and $(s,j')$.

Assume that $W$ does not satisfy the conditions given in Case 1 and Case 2.

Let $A=\{(s,t)\in W:i'<s<i, t=j~\mathrm{or}~t=j'\}$. Then $|A|\leq i-i'-1$.

First suppose that $j=0$ and $j'=1$. Also, set $A_1=\{(s,0):i\leq s\leq |C_2|-1\}$, $A_2=\{(s,1):0\leq s\leq i'\}$ and $A_3=\{(s,2):1\leq s<i\}$. Since we are not in Case 1 and Case 2, it follows that $W\subset A_1\cup A_2\cup A_3\cup A$. Therefore,
\begin{eqnarray*}
|W|&\leq & |A_1|+|A_2|+|A_3|+|A| \leq (|C_2|-i)+(i'+1)+(i-1)+(i-i'-1)\\
&=& |C_2|+i-1< 2|C_2|-1.
\end{eqnarray*}
The last inequality follows since $i<|C_2|$. This contradicts the fact that $|W|=2|C_2|-1$.

Now suppose that $j=0$ and $j'=2$ and set $A_1=\{(s,0):i\leq s\leq |C_2|-1\}$, $A_2=\{(s,1):0\leq s\leq |C_2|\}$, $A_3=\{(s,2):1\leq s\leq i'\}$.
We claim that $|A_2\cap W|\leq |A_2|-2$, unless there is a path in $G$ connecting $v$ and $w$. Indeed, we may assume that either $(i,1)$ or $(i',1)$ belongs to $W$. If both vertices belong to $W$, then there is a path in $G$ between $v$ and $w$. Without loss of generality, we may  assume that $(i,1)\in W$ and $(i',1)\notin W$. In that case, non of the elements in the nonempty set $\{(s,1):s<i'\}$ belongs to $W$, because otherwise we get a path connecting $v$ and $w$ in $G$. This proves the claim.

Since we are not in Case 1 and Case 2, we have $W\subset A_1\cup (A_2\cap W)\cup A_3\cup A$. As we may assume that $|A_2\cap W|\leq |A_2|-2$, we get
\begin{eqnarray*}
|W|&\leq & |A_1|+|A_2\cap W|+|A_3|+|A| \leq (|C_2|-i)+(|C_2|-1)+(i')+(i-i'-1)\\
&=& 2|C_2|-2,
\end{eqnarray*}
a contradiction.

Finally, $j=1$ and $j'=2$ is similarly treated as the case $j=0$ and $j'=1$.
\end{proof}

As a straightforward consequence of the next result it can be seen that the implication (a)\implies (c) in Theorem~\ref{hibi} is in general not valid for non-planar lattices, not even for hyper-planar lattices.

\begin{Theorem}
\label{longchain}
Let $P=P_1\union P_2$ be a finite poset with the property that the elements of $P_1$ and $P_2$ are incomparable to each other, and suppose that $P_2$ is a chain of length $r$. Let $L$ be the ideal lattice of $P$. Then $L$ is level for all $r\gg 0$.
\end{Theorem}

\begin{proof}
Let $L_1$ by the ideal lattice of $P_1$, and $L_2$ that of $P_2$. It is observed in \cite{HHR}   and easy to see that $K[L]\iso K[L_1]*K[L_2]$ (which is the Segre product of $K[L_1]$ and $K[L_2]$). By \cite[Theorem 4.3.1]{GW} we have $\omega_L\iso \omega_{L_1}*\omega_{L_2}$,  where for graded $K[L_1]$-module $M$ and a graded $K[L_2]$-module $N$ the homogeneous components of the Segre product $M*N$ are given by $(M*N)_i=M_i\tensor_KN_i$ for all $i$. Now if $L_2$ is a chain of length $r$, then $S=K[L_2]$ is a polynomial ring of dimension  $r+2$, and hence $\omega_{L_2}\iso S(-r-2)$, see for example  \cite[Proposition 3.6.11 and Example 3.6.10]{BH}. Hence we obtain
\begin{eqnarray}
\label{truncation}
\omega_L\iso \omega_{L_1}*S(-r-2)\iso (\omega_{L_1})_{\geq r+2},
\end{eqnarray}
where for a graded module $M$ and any integer $s$  we set $M_{\geq s}=\Dirsum_{i\geq s}M_i$. Let $t$ be the highest degree of a generator in a minimal set of generators of $\omega_{L_1}$. Then (\ref{truncation}) implies that $\omega_{L_1}$ is generated in the single degree $r+2$ if $r+2\geq t$. Thus for any $r\geq t-2$ we see that  $L$ is level.
\end{proof}

We would like to mention that the arguments given in the proof of Theorem~\ref{longchain} yield the following slightly more general result: for an arbitrary finite poset $P$ we set  $\gamma(P)=\max\{v(-\infty)\: \; v\in \MT_{0}(\hat{P})\}$. Note that $\gamma(P)$ is the highest degree of a generator in a minimal set of generators of the canonical module of the ideal lattice of $P$. Now let $P=P_1\union P_2$ and suppose that the elements of $P_1$ and $P_2$ are incomparable. Furthermore, assume that the ideal lattice $L_2$ of $P_2$ is level. Then the ideal lattice $L$ of $P$ is level if $\gamma(P_2)\geq \gamma(P_1)$.

\medskip
Computational  evidence leads us to conjecture that the equivalent conditions given in Theorem~\ref{hibi} do hold for any planar lattice (without any regularity assumption). In support of this conjecture we have the following result.

\begin{Theorem}
\label{diagonalposet}
Let $L$ be a simple planar lattice whose poset $P$ of join-irreducible elements has the single diagonal $x\gtrdot y$ with respect to a canonical chain decomposition. Then the following conditions are equivalent:
\begin{itemize}
\item[(a)] $L$ is level;
\item[(b)] $\height(x)+\depth(y)  \leq \rank \hat{P}+1$;
\item[(c)] either  $\depth(y) = \depth(x)+1$  or $\height(x) = \height(y)+1$.
\end{itemize}
\end{Theorem}

Observe that for all covering pairs $u,v\in P$ which are different from the diagonal, conditions (b) and (c) are always satisfied.

\begin{proof}[Proof of Theorem~\ref{diagonalposet}]
(a) $\Rightarrow$ (b) follows from Theorem~\ref{alsoviviana}.

(b) $\Rightarrow$ (c): Let $C_1\cup C_2$ be a canonical chain decomposition of $P$ and $x\gtrdot y$ its unique diagonal with $x\in C_1$ and $y\in C_2$. We define the integers  $a=|\{z\in \hat{P}:z>x\}|$, $b=|\{z\in \hat{P}: z<x,z\notin C_2\}|$, $c=|\{z\in \hat{P}:z>y, z\notin C_1\}|$ and $d=|\{z\in \hat{P}:z<y\}|$, see Figure~\ref{diagonal}. Then $\height (y)=d$ and $\depth (x)=a$. On the other hand,
$\depth (y)=\max\{c, a+1\}$ and $\height (x)=\max\{b, d+1\}$.
\begin{figure}[hbt]
\begin{center}
\psset{unit=0.8cm}
\begin{pspicture}(-5,-2)(4,6)

\rput(-1.5,-1){

\rput(0,0){$\bullet$}
\rput(0,2.9){$\bullet$}
\rput(0,5){$\bullet$}

\psline(1,-1)(0,0)
\psline(1,-1)(2,0)
\psline(1,6)(0,5)
\psline(1,6)(2,5)

\rput(1,-1){$\bullet$}
\rput(1,6){$\bullet$}

\rput(2,0){$\bullet$}
\rput(2,1.9){$\bullet$}
\rput(2,5){$\bullet$}

\psline(0,2.9)(2,1.9)

\rput(-0.65,1.43){$b$}
\rput(-0.65,4.3){$a$}
\rput(2.65,4.3){$c$}
\rput(2.65,1.43){$d$}
\rput(0.3,3.05){$x$}
\rput(1.7,1.75){$y$}
\rput(1,-1.3){$-\infty$}
\rput(1,6.3){$\infty$}

\psline[linestyle=dotted](0,0)(0,5)
\psline[linestyle=dotted](2,0)(2,5)
}
\end{pspicture}
\end{center}
\caption{}\label{diagonal}
\end{figure}

Suppose that none of the equalities in condition (c) hold. Then $\depth (y)>\depth (x)+1$ and $\height (x)> \height (y)+1$. So, $\max \{c,a+1\}>a+1$, and hence $c>a+1$ and  $\depth (y)=c$. Similarly, $\max \{b,d+1\}>d+1$, and hence $b>d+1$ and $\height (x)=b$. Note that $\hat{P}$ has just three maximal chains whose  lengths are $a+b$, $c+d$ and $a+d+1$,  so that $\rank \hat {P}=\max  \{a+b, c+d\}$. Thus (b) implies  $b+c\leq \max \{a+b+1 ,  c+d+1\}$,  a contradiction. Therefore one of the desired equalities in (c) hold for $x,y$.

(c)\implies (a):
Let
\begin{eqnarray*}
C_{1} &:& - \infty = x'_{b} < x'_{b-1} < \cdots < x'_{1} < x
< x_{1} < \cdots < x_{a-1} < x_{a} = \infty, \\
C_{2} &:& - \infty = y'_{d} < y'_{d-1} < \cdots < y'_{1} < y
< y_{1} < \cdots < y_{c-1} < y_{c} = \infty.
\end{eqnarray*}
The condition (c) guarantees that either $a + 1 \geq c$ or $b \leq d + 1$.
We may assume that $a + 1 \geq c$. The case  $b \leq d + 1$ is treated similarly by replacing $P$ by $P^{\vee}$.

Case 1. Suppose that a longest chain of $\hat{P}$ is
\[
- \infty = y'_{d} < y'_{d-1} < \cdots < y'_{1} < y < x <
x_{1} < \cdots < x_{a-1} < x_{a} = \infty.
\]
In other words, one has $b \leq d + 1$.
The rank of $\hat{P} = a + d + 1$.
Let $v \in\MT(\hat{P})$
with $v(-\infty) > a + d + 1$. We distinguish several cases:

\medskip
(i) Let $v(x) > a$.  Then there are $i$ and $j$ with
$v(x_{i}) - v(x_{i+1}) \geq 2$ and $v(y_{j}) - v(y_{j+1}) \geq 2$,
where $x_{0} = x$ and $y_{0} = y$.
Let $u \in \MS(\hat{P})$
with $u(z) = 1$ if either $z \leq x_{i}$ or $z \leq y_{j}$
and $u(z) = 0$, otherwise.  Then $v - u\in \MT(\hat{P})$.

\medskip
(ii) Let $v(x)=a$. In this case $v(y)\geq a+1$.

If $v(y) = a + 1$, then there are $i$ and $j$ with
$v(x'_{i+1}) - v(x'_{i}) \geq 2$ and $v(y'_{j+1}) - v(y'_{j}) \geq 2$,
where $x'_{0} = x$ and $y'_{0} = y$.
Let $u \in\MS(\hat{P})$ with
with $u(z) = 1$ if either $z \leq x'_{i+1}$ or $z \leq y'_{j+1}$
and $u(z) = 0$, otherwise.  Then $v - u\in\MT(\hat{P})$.

If $v(y) > a + 1$, then there are $i$ and $j$ with
$v(y_{j}) - v(y_{j+1}) \geq 2$
and $v(x'_{i+1}) - v(x'_{i}) \geq 2$,
where $x'_{0} = x$ and $y_{0} = y$.
Let $u\in\MS(\hat{P})$
with $u(z) = 1$ if either $z \leq x'_{i+1}$ or $z \leq y_{j}$
and $u(z) = 0$, otherwise.  Then $v - u\in\MT(\hat{P})$.

\bigskip
 Case 2. Suppose that a longest chain of $\hat{P}$ is
\[
- \infty = x'_{b} < x'_{b-1} < \cdots < x'_{1} < x <
x_{1} < \cdots < x_{a-1} < x_{a} = \infty.
\]
In other words, one has $b \geq d + 1$.
The rank of $\hat{P} = a + b$.

Let $v \in  \MT(\hat{P})$
with $v(-\infty) > a + b$. We distinguish several cases:

\medskip
(i) Let $v(x) > a$.  Then there are $i$ and $j$ with
$v(x_{i}) - v(x_{i+1}) \geq 2$ and $v(y_{j}) - v(y_{j+1}) \geq 2$,
where $x_{0} = x$ and $y_{0} = y$.
Let $u \in\MS(\hat{P})$
with $u(z) = 1$ if either $z \leq x_{i}$ or $z \leq y_{j}$
and $u(z) = 0$, otherwise.  Then $v - u\in \MT(\hat{P})$.

\medskip
(ii) Let $v(x) = a$.
Then there is $i$ with
$v(x'_{i+1}) - v(x'_{i}) \geq 2$,
where $x'_{0} = x$.
Since $b \geq d + 1$,
one has either
$v(y) - v(x) \geq 2$ or
there is $j$ with
$v(y'_{j+1}) - v(y'_{j}) \geq 2$,
where $y'_{0} = y$.

If $v(y) - v(x) \geq 2$, then there is $j$ with
$v(y_{j}) - v(y_{j+1}) \geq 2$, where $y_{0} = y$.
Let $u\in \MS(\hat{P})$
with $u(z) = 1$ if either $z \leq x'_{i+1}$ or $z \leq y_{j}$
and $u(z) = 0$ otherwise.  Then $v - u\in  \MT(\hat{P})$.

Suppose that there is $j$ with
$v(y'_{j+1}) - v(y'_{j}) \geq 2$,
where $y'_{0} = y$.
Let $u\in \MS(\hat{P})$
with $u(z) = 1$ if either $z \leq x'_{i+1}$ or $z \leq y'_{j+1}$
and $u(z) = 0$, otherwise.  Then $v - u\in \MT(\hat{P})$.

\medskip
The discussions in both cases show that every
$v\in \MT(\hat{P})$
which belongs to ${\mathcal T}_{0}(\hat{P})$
satisfies $v(-\infty) = \rank(\hat{P})$.
Hence $L$ is level, as desired.
\end{proof}

\texttt{}\begin{figure}[hbt]
\begin{center}
\psset{unit=0.8cm}
\begin{pspicture}(-5,0)(4,6)

\rput(-4,0){
\rput(3.4,3.6){$s$}
\psline(0,0)(7,0)
\psline(7,0)(7,5)
\psline(0,0)(0,2.2)
\psline(0,2.2)(3.7,2.2)
\psline(3.7,2.2)(3.7,5)
\psline(3.7,5)(7,5)
\rput(3.5,-0.3){$n$}
\rput(7.3,2.5){$m$}
\rput(1.85,2.5){$t$}
}
\end{pspicture}
\end{center}
\caption{}\label{ladder with one corner}
\end{figure}

\begin{Corollary}
\label{onecorner}
Let $I$ be the ladder determinantal ideal of 2-minors of the ladder as displayed in Figure~\ref{ladder with one corner}. Then $S/I$ is level if and only if $\min\{m,n\}\leq s+t$.
\end{Corollary}

\begin{proof}
First note that the ideal lattice of the poset with one diagonal depicted in Figure~\ref{diagonal} can be identified with the one-sided ladder as shown in Figure~\ref{ladder with one corner} with $n=c+d-1$, $m=a+b-1$, $s=a$ and $t=d$. Thus, by Theorem~\ref{diagonalposet}, $S/I$ is level if and only if condition (c) of Theorem~\ref{diagonalposet} holds. According to Figure~\ref{diagonal}, this is equivalent to say that $b\leq d+1$ or $c\leq a+1$, or equivalently $m\leq s+t$ or $n\leq s+t$, which implies the assertion.
\end{proof}

\section{Generalized Hibi rings}
\label{generalized}

Let  $P=\{x_1,\ldots,x_n\}$  be a finite poset and  $r$  a positive integer. In Section~\ref{pseudo} we identified the Hibi ring $K[L]$ as a subring of the polynomial ring $K[s,t_x\: x\in P]$. There is a different natural embedding of $K[L]$ into a polynomial ring, namely into the polynomial ring $K[s_x,t_x\:x\in P]$ where $P$ is the set of join-irreducible elements of $L$, see \cite{HH}. For this embedding the generators of $K[L]$ are  the monomials $u_\alpha=\prod_{x\in \alpha}t_x\prod_{x\in P\setminus\alpha}s_x$. This suggests an extension of  the notion of Hibi rings  as introduced in \cite{EHM}, see also \cite{EH}.

An  {\em $r$-multichain} of $\MI(P)$ \index{multichain} is a chain of
poset ideals of length $r$,
\[
\MI: \emptyset=\alpha_0\subseteq \alpha_1\subseteq \alpha_2\subseteq \cdots \subseteq \alpha_r=P.
\]
We define a partial order on the set $\MI_r(P)$ of all $r$-multichains of $\MI(P)$ by setting $\MI\leq \MI'$ if $\alpha_k\subseteq  \alpha_k^\prime$ for
$k=1,\ldots,r$.
Observe that the partially ordered set $\MI_r(P)$ is a distributive lattice, with  meet and join defined as follows: for
$\MI\: \alpha_0\subseteq \alpha_1\subseteq \cdots \subseteq \alpha_r$  and $\MI^\prime\: \alpha^\prime_0\subseteq \alpha^\prime_1\subseteq \cdots \subseteq \alpha^\prime_r$ we let $\MI\wedge \MI^\prime$ be the multichain with     $(\MI\wedge \MI^\prime)_k=\alpha_k\sect \alpha_k^\prime$ for $k=1,\ldots,r$, and
 $\MI\vee \MI^\prime$ the multichain with $(\MI\vee \MI^\prime)_k=\alpha_k\union \alpha^\prime_k$ for $k=1,\ldots,r$.

With each $r$-multichain $\MI$ of $\MI_r(P)$ we associate a monomial $u_{\MI}$ in the polynomial ring $K[x_{ij} : 1\leq i\leq r,
1\leq j\leq n]$ in $rn$ indeterminates which is defined as
\[
u_\MI=x_{1\gamma_1}x_{2\gamma_2}\cdots x_{r\gamma_r},
\]
where $x_{k\gamma_k}=\prod_{x_\ell\in \gamma_k}x_{k\ell}$ and $\gamma_k=\alpha_k\setminus \alpha_{k-1}$\quad for $k=1,\ldots,r$.

The {\em generalized Hibi ring} $\MR_r(P)$ is the toric ring generated by the  monomials $u_\MI$ with $\MI\in \MI_r(P)$. By what we said at the beginning of this section it is clear that $K[L]=\MR_2(P)$.

In \cite[Theorem 4.1]{EHM} is shown that $\MR_r(P)$ can be identified with the ordinary Hibi ring $K[L_r]$ where $L_r$ is the ideal lattice of $\MI_r(P)$, and further it is shown in \cite[Theorem 4.3]{EHM} that $\MI_r(P)$ is isomorphic to the poset $P_r=P\times Q_{r-1}$ where $Q_{r-1}=[r-1]$ with the natural order of its elements, and where $P\times Q_{r-1}$ denotes the direct product of the posets $P$ and $Q_{r-1}$.  In general the direct product $P\times Q$ of two posets $P$ and $Q$ is defined to be the poset which as a set is just the cartesian product of two sets $P$ and $Q$ and with partial order given by $(x,y)\leq (x',y')$ if and only if $x\leq x'$ and $y\leq y'$.

This identification of $\MR_r(P)$ with $K[L_r]$ was used in \cite{EHM} to prove that for any given $r\geq 2$, the Hibi ring $K[L]$ is Gorenstein if and only if $K[L_r]$ is Gorenstein.

We denote by $\type(R)$ the {\em Cohen--Macaulay type} of a Cohen--Macaulay ring  $R$. It is defined to be the number of generators of $\omega_R$.   Here we show

\begin{Theorem}
\label{new}
Let $L$ be a finite distributive lattice, and $r\geq 2$ an integer. Then
\begin{enumerate}
\item[{\em (a)}] $\type(K[L])\leq \type(K[L_r])$;
\item[{\em (b)}] $L$ is pseudo-Gorenstein if and only if $L_r$ is pseudo-Gorenstein;
\item[{\em (c)}] If $L_r$ is level then  $L$ is level.
\end{enumerate}
\end{Theorem}

\begin{proof}
(a) We need to show that $|\MT_0(\hat{P})|\leq |\MT_0(\hat{P}_r)|$. In order to prove this we define an injective map $\iota \: \MT_0(\hat{P})\to \MT_0(\hat{P}_r)$. Given $v\in \MT_0(\hat{P})$ we let $\iota(v)(x,i)=v(x)+(r-1-i)$ and $\iota(v)(\infty)=0$ and $\iota(v)(-\infty)=v(-\infty)+(r-2)$. Obviously, $\iota(v)\neq \iota(w)$ for $v, w\in \MT_0(\hat{P})$ with $v\neq w$ and $\iota(v)\in \MT(\hat{P}_r)$. Thus it remains to show that $\iota(v)$ actually belongs to $\MT_0(\hat{P_r})$. We set $v'=\iota(v)$, and show that if $v'-u\in  \MS(\hat{P}_r)$ for some $u\in \MT(\hat{P}_r)$, then $v'=u$.

For any $w\in \MT(\hat{P}_r)$ and for $i\in [r-1]$ we define the function  $w_i$ on $\hat{P}$ as follows:
\[
w_i(x)=w(x,i)-(r-1-i) \quad \text{for all} \quad x\in P,
\]
$w_i(\infty)=0$ and $w_i(-\infty)=\max\{w_i(x)\: x\in P\}+1$.  Then $w_i\in \MT(\hat{P})$.

Since $v'-u\in  \MS(\hat{P}_r)$ it follows that  $v-u_i=v'_i-u_i\in \MS(\hat{P})$. Since $v\in \MT_0(\hat{P})$ we see that $v=u_i$ for all $i$. This shows that $v'=u$.

(b) Let $x\in P$ and $i\in [r-1]$. We claim that $\height_{\hat{P}_r}(x,i)=\height_{\hat{P}}(x)+(i-1)$ and $\depth_{\hat{P}_r}(x,i)=\depth_{\hat{P}}(x)+(r-i-1)$. If $\height_{\hat{P}_r} (x,i)=1$, then there is nothing to prove. Let $\height_{\hat{P}_r} (x,i)>1$ and let $x=x_0>x_1>\cdots>x_d>-\infty$ be a  maximal chain of length $\height_{\hat{P}}(x)$ in $\hat{P}$. Then
$(x,i)=(x_0,i)>(x_1,i)>\cdots>(x_d,i)>(x_d,i-1)>\cdots>(x_d,1)>-\infty$ is a maximal chain of length $\height_{\hat{P}}(x)+(i-1)$ in ${\hat{P}_r}$. It follows that
$\height_{\hat{P}_r}(x,i)\geq \height_{\hat{P}}(x)+(i-1)$. To prove the other inequality we use induction on height. Let $(x,i)=z_0>z_1>\cdots >z_t>-\infty$ be a maximal chain of length $\height_{\hat{P}_r}(x,i)$ in $\hat{P}_r$. Then $z_1=(x,i-1)$ or $z_1=(x',i)$ where $x'$ is an element of $P$ covered by $x$. If $z_1=(x,i-1)$, then by induction hypothesis, we get $\height_{\hat{P}_r} (x,i-1)\leq \height_{\hat{P}} (x)+(i-2)$, and hence $\height_{\hat{P}_r} (x,i)=\height_{\hat{P}_r} (x,i-1)+1\leq \height_{\hat{P}} (x)+(i-1)$. If $z_1=(x',i)$, then our induction hypothesis implies that $\height_{\hat{P}_r} (x',i)\leq \height_{\hat{P}} (x')+(i-1)$, and hence $\height_{\hat{P}_r} (x,i)=\height_{\hat{P}_r} (x',i)+1\leq \height_{\hat{P}} (x)+(i-1)$, where the last inequality follows from the fact that
$\height_{\hat{P}} (x')\leq \height_{\hat{P}} (x)-1$. So, $\height_{\hat{P}_r}(x,i)=\height_{\hat{P}}(x)+(i-1)$. A similar argument can be applied to prove the claimed  formula regarding the depth. As a side result, we obtain that $\rank \hat{P}_r=\rank \hat{P}+(r-2)$.

It follows that $\height_{\hat{P}}(x)+\depth_{\hat{P}}(x)=\rank \hat{P}$ if and only if $\height_{\hat{P}_r}(x,i)+\depth_{\hat{P}_r}(x,i)=\rank \hat{P}_r$. Thus Theorem~\ref{classification} yields  the desired result.

(c) Suppose  that $L$ is not level. Then there exists $v\in \MT_0(\hat{P})$ with $v(-\infty)>\rank \hat{P}$. Then $\iota(v)$, as defined in the proof of part (a), belongs to  $\MT_0(\hat{P}_r)$ and
\[
\iota(v)(-\infty)=v(-\infty)+(r-2)>\rank \hat{P}+(r-2)=\rank \hat{P}_r.
\]
This shows that $L_r$ is not level.
\end{proof}

Note that if $\type(K[L])=1$, then $\type(K[L_r])=1$ for all $r\geq 2$, since by \cite[Corollary~4.5]{EHM}, $K[L]$ is Gorenstein if and only if $K[L_r]$ is Gorenstein. But the following example given in Figure~\ref{type} shows that the inequality in Theorem~\ref{new} (a) may be strict. Indeed, let $P$ be the poset depicted in the left side of Figure~\ref{type}. Then $P_3$ is the poset which is shown in right side of Figure~\ref{type}. It can be easily checked, by considering all possible strictly  order reversing functions on $\hat{P}$ and $\hat{P}_3$ which correspond to the minimal generators of the canonical module, that $\type(K[L])=2$, but $\type(K[L_3])=3$.
\begin{figure}[hbt]
\begin{center}
\psset{unit=0.8cm}
\begin{pspicture}(-10.5,2)(4,5)

\rput(-1.5,2){
\psline(-6,0)(-6,1.5)
\rput(-6,0){$\bullet$}
\rput(-6,1.5){$\bullet$}

\rput(-4.5,0){$\bullet$}


\psline(1,0)(0,1)
\rput(1,0){$\bullet$}
\rput(0,1){$\bullet$}

\psline(1,0)(2,1)
\rput(1,0){$\bullet$}
\rput(2,1){$\bullet$}

\psline(0,1)(1,2)
\rput(0,1){$\bullet$}
\rput(1,2){$\bullet$}

\psline(2,1)(1,2)
\rput(2,1){$\bullet$}
\rput(1,2){$\bullet$}

\psline(3.5,0)(3.5,1.5)
\rput(3.5,0){$\bullet$}
\rput(3.5,1.5){$\bullet$}
}
\end{pspicture}

\end{center}
\caption{}\label{type}
\end{figure}

 In Theorem~\ref{new}(c) we only could  prove that if $L_r$ is level, then $L$  is level. However we expect
that the other  implication also holds. An indication that this might be true  is Miyazaki's theorem and the easy to prove fact that  a poset $P$ is Miyazaki
if and only if $P_r$ is Miyazaki.  Moreover, if the necessary condition  for levelness given in Theorem~\ref{alsoviviana} would also be sufficient, which
indeed we expect, then one could also conclude that $L$ is level if and only if $L_r$ is level, since $\hat{P}$ satisfies the inequalities (\ref{inequality})
if and only if this is the case for $\hat{P}_r$.

\end{document}